\documentclass[12pt,reqno]{amsart}

\usepackage{amsfonts,amsthm,amsmath,microtype}
\newcommand{\numberset}{\mathbb} 
\newcommand{\R}{\numberset{R}} 
\newcommand{\fe}{f_\varepsilon}
\newcommand{\varphie}{\varphi_\varepsilon}
\usepackage[a4paper,top=3cm,bottom=3cm,left=3cm,right=3cm,bindingoffset=0mm]{geometry}

\usepackage{verbatim} 
\usepackage{setspace}

\makeatletter

\numberwithin{equation}{section}
\newtheorem{thm}{\indent\bf {Theorem}}[section]

\newtheorem{prop}[thm]{\indent\bf Proposition}

\theoremstyle{definition}
\newtheorem{exmpl} [thm] {\indent \textsc {Example}}
\newtheorem{rmk}[thm]{\indent \textsc {Remark}}

\def\author@andify{%
   \nxandlist {\unskip ,\penalty-1 \space\ignorespaces}%
     {\unskip {} }%
     {\unskip ,\penalty-2 \space }%
}

\begin{document}

\title[A class of weighted Hardy inequalities\\
 and applications to evolution problems]
{A class of weighted Hardy inequalities\\
and applications to evolution problems}

\author[A. Canale]{Anna Canale}
\address{Dipartimento di Ingegneria dell'Informazione ed Elettrica e Matematica Applicata, 
Universit\'a degli Studi di Salerno, Via Giovanni Paolo II, 132, 84084 Fisciano
(Sa), Italy.}

\author[F.Pappalardo]{Francesco Pappalardo}
\address{Dipartimento di Matematica e Applicazioni \textquotedblleft Renato  Caccioppoli\textquotedblright,
Universit\'a degli Studi di Napoli Federico II, Complesso Universitario Monte S. Angelo, Via Cintia, 80126 Napoli, Italy.}

\author[C. Tarantino]{Ciro Tarantino}
\address{Dipartimento di Scienze Economiche e Statistiche, 
Universit\'a degli Studi di Napoli Federico II, Complesso Universitario Monte S. Angelo,
 Via Cintia, 80126 Napoli, Italy.}

\thanks{The first two authors are members of the Gruppo Nazionale per l'Analisi Matematica, la Probabilit\'a e le loro Applicazioni 
(GNAMPA) of the Istituto Nazionale di Alta Matematica (INdAM)}


\begin{abstract}
We state the following weighted Hardy inequality
\begin{equation*}
c_{o, \mu}\int_{{\R}^N}\frac{\varphi^2 }{|x|^2}\, d\mu\le  
\int_{{\R}^N} |\nabla\varphi|^2 \, d\mu +
 K \int_{\R^N}\varphi^2 \, d\mu \quad \forall\, \varphi \in H_\mu^1
\end{equation*}
in the context of the study of the Kolmogorov operators
\begin{equation*}
Lu=\Delta u+\frac{\nabla \mu}{\mu}\cdot\nabla u
\end{equation*} perturbed by inverse square potentials and of the related evolution problems.
The function $\mu$ in the drift term is a probability density on $\R^N$. 
We prove the optimality of the constant $c_{o, \mu}$
and state existence and nonexistence results following the Cabr\'e-Martel's approach
\cite{CabreMartel} extended to Kolmogorov operators.


\end{abstract}

\maketitle


\bigskip

\section{Introduction}
This paper on weighted Hardy inequalities fits in the framework 
of the study of Kolmogorov operators
on smooth functions
\begin{equation*}
Lu=\Delta u+\frac{\nabla \mu}{\mu}\cdot\nabla u,
\end{equation*}
with $\mu$ probability density on $\R^N$, 
and of the related evolution problems
$$
(P)\quad \left\{\begin{array}{ll}
\partial_tu(x,t)=Lu(x,t)+V(x)u(x,t),\quad \,x\in {\mathbb R}^N, t>0,\\
u(\cdot ,0)=u_0\geq 0\in L_\mu^2.
\end{array}
\right. $$
The operator $L$ in $(P)$ is perturbed by the singular potential  $V(x)=\frac{c}{|x|^2}$, $c>0$,
and $L_\mu^2:=L(\R^N, d\mu)$,  with $d\mu(x)=\mu(x)dx$.

The interest in inverse square potentials of type $V\sim\frac{c}{|x|^2}$
relies in their criticality: the strong maximum principle and Gaussian bounds
fail (see \cite{Aronson}).
Furthermore interest in singular potentials is due to the applications to many fields, for example 
in many physical contexts as molecular physics
\cite{Levy}, 
quantum cosmology 
(see e.g. \cite{BerestyckiEsteban}), quantum mechanics \cite{BarasGoldstein2} 
and combustion models \cite{Gelfand}.

The operator $\Delta+V$,  $V(x)=\frac{c}{|x|^{2}}$,
has the same homogeneity as the Laplacian and do not belong to the Kato's class,
then $V$ cannot be regarded as a lower order perturbation term.

A remarkable result stated in 1984 by P. Baras and J. A. Goldstein in \cite{BarasGoldstein}
shows that the evolution problem $(P)$ with $L=\Delta$
admits a unique positive solution
if $c\leq c_o=\left( \frac{N-2}{2} \right)^{2}$ and no positive solutions exist if $c>c_o$.
When it exists, the solution is
exponentially bounded, on the contrary, if $c>c_o$, there is the so-called instantaneous blowup phenomenon.

In order to extend these results to Kolmogorov operators the technique must be different.

A result analogous to that stated in  \cite{BarasGoldstein} has been obtained in 1999 by 
X. Cabr\'e and Y. Martel \cite{CabreMartel}  for more general potentials 
$0\le V\in L_{loc}^1(\R^N)$ with a different approach.

To state the existence and nonexistence results we follow the Cabr\'e-Martel's approach.
We use the relation between the weak solution of $(P)$
and the {\it bottom of the spectrum} of the operator $-(L+V)$
\begin{equation*}
\lambda_1(L+V):=\inf_{\varphi \in H^1_\mu\setminus \{0\}}
\left(\frac{\int_{{\mathbb R}^N}|\nabla \varphi |^2\,d\mu
-\int_{{\mathbb R}^N}V\varphi^2\,d\mu}{\int_{{\mathbb R}^N}\varphi^2\,d\mu}
\right)
\end{equation*}
with $H^1_\mu$ suitable weighted Sobolev space.

When $\mu=1$ Cabr\'e and Martel 
showed that the boundedness  of 
 $\lambda_1(\Delta+V)$, 
$0\le V\in L_{loc}^1({\mathbb R}^N)$, is a necessary 
and sufficient condition for the existence of positive exponentially bounded in time
solutions to the associated initial value problem. 
Later in \cite{GGR, CGRT} similar results have been extended to Kolmogorov operators.
The proof uses some properties of the operator $L$ and of its corresponding semigroup
in $L_\mu^2(\R^N)$.

For Ornstein-Uhlenbeck type operators 
$Lu=\Delta u - \sum_{i=1}^{n}A(x-a_i)\cdot \nabla u$,
$a_i\in \R^N$, $i=1,\dots , n$, perturbed by multipolar
inverse square potentials, weighted multipolar Hardy inequalities
and related existence and nonexistence results were stated in \cite{CP}. 
In such a case, the invariant measure for these operators is
$d\mu =\mu_A (x) dx =Ke^{-\frac{1}{2}\sum_{i=1}^{n}\left\langle A(x-a_i), x-a_i\right\rangle }dx$.  

There is a close relation between
the estimate of the bottom of the spectrum $\lambda_1(L+V)$ and the weighted Hardy inequality 
with $V(x)=\frac{c}{|x|^2}$, $c\le c_{o,\mu}$,  
\begin{equation}\label{whiIntro}
\int_{\R^N}V\,\varphi^2\,d\mu
\leq\int_{\R^N}|\nabla \varphi |^2d\mu+
K\int_{\R^N} \varphi^2d\mu \quad \forall\,\varphi\in H^1_\mu,\qquad K>0
\end{equation}
with the best possible constant $c_{o,\mu}$.

In particular the existence of positive solutions to $(P)$
is related to the Hardy inequality (\ref{whiIntro})  and the nonexistence 
is due to the optimality of the constant $c_{o, \mu}$. 

The main results in the paper are,  in Section 2, the weighted Hardy inequality (\ref{whiIntro}) with 
measures which satisfy fairly general conditions and the optimality 
of the constant $c_{o, \mu}$ in Section 3. 

The proof of the weighted Hardy inequality is different from the others in literature.
It is based on the introduction of a suitable $C^\infty$-function and it can be used 
to prove inequality (\ref{whiIntro}) with $0\le V\in L^1_{loc}(\R^N)$ of a more general type,
in other words Hardy type inequalities. 

In \cite{CGRT}
the authors state a weighted Hardy inequality using a different approach and
improved Hardy inequalities. This requires suitable conditions on $\mu$.
Our technique, with different assumptions on $\mu$,
allows us to achieve the best constant (cf. \cite[Theorem 3.3]{CGRT})
for a wide class of functions $\mu$.
To state the optimality of the constant in the estimate we need further assumptions on $\mu$ 
as usually it is done. We find a suitable function $\varphi$ for which the inequality  (\ref{whiIntro})  
doesn't hold if $c>c_{o,\mu}$, this is a crucial point in the proof. 
The way of estimate the bottom of the spectrum is close to the one used in \cite{CGRT}.
We remark that the inequality obtained under our hypotheses applies in the context of weighted 
multipolar Hardy inequalities stated in the forthcoming paper \cite{CPT2}.

Finally we state an existence and nonexistence result in Section 4
following the Cabr\'e-Martel's approach and using some results stated in 
\cite{GGR, CGRT} when the function $\mu$ belongs to 
$C^{1,\lambda}_{loc}(\R^N)$ or belongs to 
$C^{1,\lambda}_{loc}(\R^N\setminus\{0\})$, for some
$\lambda\in (0,1)$.

Some classes of functions $\mu$ satisfying the hypotheses of the main Theorems  are given
in Section 2.

\bigskip

\section{Weighted Hardy inequalities}

\bigskip

Let $\mu$ a weight function in $\R^N$. We define the weighted Sobolev space 
$H^1_\mu=H^1(\R^N, \mu(x)dx))$
as the space of functions in $L^2_\mu:=L^2(\R^N, \mu(x)dx)$ whose weak derivatives belong to
$(L_\mu^2)^N$.

As first step we consider the following conditions on $\mu$ which we need to state 
a preliminary weighted Hardy inequality.

\begin{itemize}
\item[$H_1)$] $\quad\mu\ge 0$, $\mu\in L^1_{loc}(\R^N)$;
\item[$H_2)$] $\quad\nabla \mu \in L_{loc}^1(\R^N)$;
\item[$H_3)$] $\quad$ there exist constants $k_1, k_2\in \R$, $k_2>2-N$, such that if
\begin{equation*}
\fe=(\varepsilon +|x|^{2})^{\frac{\alpha}{2}}, 
\quad \alpha< 0 , \quad \varepsilon >0,
\end{equation*}
it holds
$$\frac{\nabla f_\varepsilon}{f_\varepsilon}\cdot\nabla\mu=
\frac{\alpha x}{\varepsilon +|x|^{2}} \cdot\nabla\mu
\le\left( k_1 
+ \frac{k_2\alpha}{\varepsilon +|x|^2}\right) \mu$$
for any $\varepsilon>0$. 

\end{itemize}

\medskip

\noindent The condition $H_3)$ contains the requirement that the scalar product
$\alpha x \cdot\frac{\nabla\mu}{\mu}$ is bounded in $B_R$, $R>0$,
while $\frac{\alpha x}{\varepsilon +|x|^{2}}\cdot\frac{\nabla\mu}{\mu}$
is bounded in $\R^N\setminus B_R$,
where $B_R$ is a ball of radius $R$ centered in zero. 

The reason we use the function $f_\varepsilon$, introduced in \cite{Davies}, will be clear in the proof of 
the weighted Hardy inequality which we will state below. 
Finally we observe that we need the condition $k_2>2-N$ to apply Fatou's lemma in the proof of 
Theorem \ref{Thm wHi}. 
\medskip

\begin{thm}\label{Thm wHi}
Under conditions  $H_1)$--$H_3)$ there exists a positive constant $c$  such that
\begin{equation}\label{wHiCap5iniziale}
c \int_{\R^N} \frac{\varphi^2}{|x|^2}\,d\mu \le \int_{\R^N}|\nabla \varphi|^2\,d\mu 
+k_1\int_{\R^N}\varphi^2\,d\mu,
\end{equation}
for any function $\varphi \in C^\infty_c(\R^N)$, where $c\in(0,c_o(N+k_2)]$ with
$c_o(N+k_2)=\left( \frac{N+k_2-2}{2}\right)^2 $.
\end{thm}

\begin{proof}
As first step we start from the integral of the square of the gradient of the function $\varphi$. 
Then we introduce $\psi =\frac{\varphi}{\fe}$, with $\fe$ defined in $H_3)$, and integrate by parts
taking in mind $H_1)$ and $H_2)$.          
\begin{equation}\label{phi psi}
\begin{split}
&\int_{\R^N}|\nabla \varphi|^2\, d\mu = 
\int_{\R^N}|\nabla (\psi\fe)|^2\, d\mu\\
&=
\int_{\R^N}|\nabla\psi \fe +\nabla\fe \psi|^2\, d\mu \\
&=
 \int_{\R^N}  |\nabla \psi|^2\fe^2\, d\mu +\int_{\R^N} \psi^2|\nabla \fe|^2\, d\mu 
+2\int_{\R^N}\fe\psi\nabla\psi \cdot \nabla\fe \, d\mu\\ 
&=
\int_{\R^N} |\nabla \psi|^2\fe^2\, d\mu
+\int_{\R^N}\psi^2 |\nabla \fe|^2 \, d\mu \\
&-
\int_{\R^N} \psi^2|\nabla \fe|^2 \, d\mu-
\int_{\R^N}\fe^2\psi^2 \frac{\Delta\fe} {\fe}\, d\mu-
\int_{\R^N}\fe^2\psi^2 \frac{\nabla \fe}{\fe} \cdot \nabla \mu\, dx.
\end{split}
\end{equation}
Observing that
$$\Delta f_\varepsilon=\frac{\alpha(N-2+\alpha)|x|^2+\alpha\varepsilon N}
{(\varepsilon +|x|^{2})^{2-\frac{\alpha}{2}}}$$
and using hypothesis $H_3)$ we deduce that
\begin{equation}\label{ineq in wHi}
\begin{split}
\int_{\R^N}|&\nabla \varphi|^2\, d\mu\ge
-\int_{\R^N}\frac{\Delta \fe}{\fe} \varphi^2\, d\mu
-\int_{\R^N} \frac{\nabla \fe}{\fe} \cdot \nabla \mu\,\varphi^2\, dx \\
&\ge
-\left[ \alpha (N-2)+\alpha^2\right]\int_{\R^N}\frac{|x|^2 }{(\varepsilon +|x|^2)^2}\varphi^2\, d\mu 
-\varepsilon \alpha N\int_{\R^N}\frac{\varphi^2}{(\varepsilon +|x|^2)^2}\, d\mu\\ 
&- 
k_1 \int_{\R^N}\varphi^2\,d\mu
- k_2 \alpha\int_{\R^N}\frac{\varphi^2}{\varepsilon +|x|^2}\, d\mu\\
&=
[-\alpha (N-2+k_2)-\alpha^2]\int_{\R^N}\frac{|x|^2 }{(\varepsilon +|x|^2)^2}\varphi^2\, d\mu\\
&-
\varepsilon \alpha (N+k_2)\int_{\R^N}\frac{\varphi^2}{(\varepsilon +|x|^2)^2}\, d\mu-
k_1 \int_{\R^N}\varphi^2\,d\mu.
\end{split}
\end{equation}
The constant $-\alpha (N-2+k_2)-\alpha^2$ is greater than zero since $\alpha<0$ and $k_2>2-N$,
so by Fatou's lemma we state the following estimate letting $\varepsilon\to 0$
\begin{equation*}
\int_{\R^N}|\nabla \varphi|^2\, d\mu + k_1 \int_{\R^N}\varphi^2\, d\mu \ge
c\int_{\R^N}\frac{\varphi^2}{|x|^2}\,d\mu,
\end{equation*}
with $c=-\alpha (N-2+k_2)-\alpha^2$.
Finally we observe that 
$$\max_\alpha [-\alpha (N+k_2-2)-\alpha^2] 
=\left( \frac{N+k_2-2}{2}\right)^2=:c_o(N+k_2),$$
attained for $\alpha_o= -\frac{N+k_2-2}{2}$.\qedhere
\end{proof}

\medskip

\begin{rmk}
In an alternative way we can define $\fe$ in $H_3)$ setting $\alpha=\alpha_o$
and get the estimate (\ref{wHiCap5iniziale}) with $c=c_o(N+k_2)$.  
Although the result goes in the same direction,
in the proof we point out that $c_o(N+k_2)$ is the maximum value of the constant $c$.
\end{rmk}

\medskip

\begin{rmk}
In the case $\mu=1$ we obtain the classical Hardy inequality.
We remark that if in the proof we introduce a function $f\in C^\infty(\R^N)$
in place of $\fe$, the inequality (\ref{ineq in wHi}) can be used to get 
Hardy type inequalities
\begin{equation}\label{H type ineq}
\int_{\R^N} V\varphi^2\,dx \le \int_{\R^N}|\nabla \varphi|^2\,dx
\end{equation}
where the potential $V=V(x)\in L^1_{loc}(\R^N)$, $V(x)\ge 0$, is such that
$$
-\frac{\Delta f}{f} \ge V \qquad \forall x\in \R^N.
$$
Operators perturbed by potentials of a more general type,
for which the generation of semigroups was stated,
have been investigated, for example, in  \cite{CRT1,CRT2,CT} when $\mu=1$.
For functions $\mu\ne 1$ such that $k_2\ne 0$ we have to modify the condition $H_3)$
to get the Hardy type inequality (\ref{H type ineq}) with respect to the measure $d\mu$.
\end{rmk}

\medskip

Now we suppose that

\begin{itemize}
\item[$H_4)$] $\quad\mu\ge 0$, $\sqrt{\mu}\in H^1_{loc}(\R^N)$;
\item[$H_5)$] $\quad\mu^{-1}\in L_{loc}^1(\R^N)$.
\end{itemize}
Let us observe that in the hypotheses $H_4)$-$H_5)$ the space 
$C^\infty_c(\R^N)$ is dense in $H^1_\mu$
and $H^1_\mu$ is the completion of $C^\infty_c(\R^N)$ with respect to the Sobolev norm 
$$
\|\cdot\|_{H^1_\mu}^2 := \|\cdot\|_{L^2_\mu}^2 + \|\nabla \cdot\|_{L^2_\mu}^2
$$
(see \cite{T}). For some interesting papers on density of smooth functions in weighted Sobolev spaces  
and related questions we refer, for example,  to 
 \cite{K, FKS, B, KO, CPSC, Z, Bogachev}.

So we can deduce the following result from Theorem \ref{Thm wHi} by density argument.

\medskip

\begin{thm}\label{Thm wHi2}
Under conditions  $H_2)$--$H_5)$ there exists a positive constant $c$  such that
\begin{equation}\label{wHiCap5}
c \int_{\R^N} \frac{\varphi^2}{|x|^2}\,d\mu \le \int_{\R^N}|\nabla \varphi|^2\,d\mu 
+k_1\int_{\R^N}\varphi^2\,d\mu,
\end{equation}
for any function $\varphi \in H^1_\mu$, where $c\in(0,c_o(N+k_2)]$ with
$c_o(N+k_2)=\left( \frac{N+k_2-2}{2}\right)^2 $.
\end{thm}

\medskip 

\medskip

We give some examples of functions $\mu$ which satisfy the hypotheses of \break Theorem \ref{Thm wHi2}.

We remark that, in the hypotheses $\mu=\mu(|x|)\in C^1$ for $|x|\in[r_0,+\infty[$, $r_0> 0$,   
a class of weight functions $\mu$ which satisfies $H_3)$ is the following
\begin{equation}\label{condizioni su mu}
\mu(x) \ge C e^{-\frac{k_1}{2|\alpha|}|x|^2}|x|^{k_2-\frac{k_1}{|\alpha|}\varepsilon},
\quad \text{ for}\quad |x|\ge r_0,
\end{equation}
where $C$ is a constant depending on $\mu(r_0)$ and $r_0$.

Indeed, in the case of radial functions, $\mu(x)=\mu(|x|)$, if we set $|x|=\rho$ the condition $H_3)$ states that 
$\mu$ satisfies the following inequality
$$
\frac{\alpha\rho}{\varepsilon +\rho^2}\mu'(\rho) \le \left(k_1+
\frac{  k_2 \alpha}{\varepsilon +\rho^2}\right)\mu(\rho),
$$
which implies
$$\mu'(\rho) \ge a(\rho)\mu(\rho)$$
with
$$
a(\rho)=\left[\frac{k_1}{\alpha}\left(\frac{\varepsilon +\rho^2}{\rho}\right)
+\frac{  k_2}{\rho}\right].
$$
Integrating in $[r_0, r]$ we get
$$
\mu(r)\ge \mu(r_0) e^{\int_{r_0}^r a(s)ds}=
\mu(r_0)\left(\frac{r}{r_0}\right)^{k_2-\frac{k_1}{|\alpha|}\varepsilon}
e^{-\frac{k_1}{2|\alpha|}(r^2-r_0^2)}
\qquad \text{ for}\quad r\ge r_0,
$$
from which
\begin{equation*}
\mu(r)\ge
\frac{\mu(r_0)}{ r_0^{k_2-\frac{k_1}{|\alpha|}\varepsilon}}
e^{\frac{k_1}{2|\alpha|} r_0^2}
r^{k_2 -\frac{k_1}{|\alpha|}\varepsilon}e^{-\frac{k_1}{2|\alpha|} r^2} \qquad \text{ for}\quad r\ge r_0.
\end{equation*}

\medskip

\begin{exmpl}\label{es cos}
Another class of weight functions satisfying $H_3)$, when $k_1=k_2=0$,  
consists of the bounded increasing functions,
as, for example, $\cos e^{-|x|^2}$. Such a function verifies the requirements in the Theorem \ref{Thm wHi2}.
\end{exmpl}

\noindent In the following example we consider a wide class of functions which contains the Gaussian 
measure and polynomial type measures. A class of functions which behaves as $\frac{1}{|x|^\gamma}$
when $|x|$ goes to zero.

\bigskip

\begin{exmpl}\label{es esp}
We consider the following weight functions 
\begin{equation}\label{esempio}
\mu(x)=\frac{1}{|x|^\gamma}e^{-\delta |x|^m}, \quad \delta \ge 0, \quad \gamma<N-2
\end{equation} 
and state for which values of $\gamma$  and $m$ the functions in (\ref{esempio})
are 
\break \lq\lq good\rq\rq $\,$functions to get the weighted Hardy inequality (\ref{wHiCap5}).

The weight $\mu $ satisfies $H_2)$, $H_4)$ and $H_5)$ if $\gamma>-N$. 
The condition $H_3)$ 
$$
\frac{\alpha(-\gamma-\delta m |x|^m)}{\varepsilon +|x|^2}\le k_1+\frac{\alpha k_2}{\varepsilon +|x|^2}
$$
is fulfilled if
\begin{equation}\label{disug-esempio}
-(\alpha\gamma+\alpha k_2 +k_1 \varepsilon)-\alpha\delta m |x|^m  -k_1|x|^2 \le 0.
\end{equation}
In the case $\delta=0$ we only need to require that
$\gamma\le-k_2-\frac{k_1}{\alpha}\varepsilon$ and we are able to get the Caffarelli-Niremberg inequality
\begin{equation*}
\left( \frac{N-2-\gamma}{2}\right)^2 \int_{\R^N} \frac{\varphi^2}{|x|^2}|x|^{-\gamma}\,dx 
\le \int_{\R^N}|\nabla \varphi|^2|x|^{-\gamma}\,dx \qquad \forall\varphi \in H_\mu^1.
\end{equation*}
While if $\gamma=0$ the inequality (\ref{wHiCap5}) holds,
for $k_1$ large enough, with $k_2=0$ if $m=2$ and with $k_2<0$ if $m<2$.

In general to get (\ref{disug-esempio}) we need the following conditions on 
parameters and on the constant $k_1$:

\begin{itemize}
\item[$i)$] $\qquad \gamma \in (-N,-k_2]$, $\delta=0$, $k_1=0$,
\item[$ii)$] $\qquad \gamma \in (-N,-k_2]$,  $k_1\ge-2\alpha\delta$, $m= 2$,
\item[$iii)$] $\qquad \gamma \in (-N,-k_2)$,   $k_1\ge \tilde k_1$, $m<2$,
\end{itemize}
where $\tilde k_1=\frac{
\frac{m}{2}\left(1-\frac{m}{2}\right)^{\frac{2}{m}-1}(-\alpha \delta m)^{\frac{2}{m}}
}
{[\alpha(\gamma+k_2)]^{\frac{2}{m}-1}
}$,
to get the inequality (\ref{wHiCap5}).
\end{exmpl}

\medskip

\begin{exmpl}
The function $\mu(x)=[\log (1+|x|)]^{-\gamma}$, for $\gamma<N-2$,
behaves as $\frac{1}{|x|^\gamma}$ when $|x|$ goes to 0.
So can state the weighted Hardy inequality (\ref{wHiCap5}) with $k_1=0$ 
and $\gamma \in (-N,-k_2]$ as in the previous example.
\end{exmpl}

\bigskip

\section{Optimality of the constant}

\bigskip

To state the optimality of the constant $c_o(N+K_2)$ in the estimate (\ref{wHiCap5}) 
we need further assumptions on $\mu$ as usually it is done. 
We remark that in the proof of optimality the choice 
of the function $\varphi$ plays a fundamental role.

\medskip

We suppose
\medskip
\begin{itemize}
\item  [$H_6)$] $\quad$
$\frac{\mu(x)}{|x|^\delta} 
\in L^1_{loc}(\R^N)$ iff
$\delta\le N+k_2$.
\end{itemize}

\medskip

\noindent We observe that the condition $H_6)$  is necessary for the technique used 
to estimate the bottom of the spectrum of the operator $-L-V$
in the proof of the optimality. 
For example the functions $\mu$ such that
$$
\lim_{|x|\to 0}\frac{\mu(|x|)}{|x|^{k_2}}=l, \qquad l>0,
$$
verify $H_6)$.

The result below states the optimality of the constant $c_o(N+k_2)$ in the Hardy inequality.

\medskip

\begin{thm}\label{thm-opt}
In the hypotheses of Theorem \ref{Thm wHi2} and if $H_6)$ holds,
 for \break $c>c_o(N+k_2)=\left( \frac{N+k_2-2}{2} \right)^2$ 
the inequality (\ref{wHiCap5}) doesn't hold for any $\varphi \in H_\mu^1$.
\end{thm}

\begin{proof}
Let $\theta \in C_c^\infty(\R^N)$ a cut-off function, $0 \le \theta \le 1$, 
$\theta=1$ in $B_1$ and $\theta=0$ in $B_2^c$.
We introduce the function
$$
\varphie(x)= \left\{\begin{array}{ll}
(\varepsilon+|x|)^\eta \quad &\text{ if } |x|\in [0,1[,\\
(\varepsilon+|x|)^\eta \theta(x) \quad &\text{ if } |x|\in [1,2[,\\
0 \quad &\text{ if } |x|\in [2,+\infty[,
\end{array}
\right. $$
where $\varepsilon >0$ and the exponent $\eta$ is such that
$$
\max \left\lbrace -\sqrt{c},-\frac{N+k_2}{2} \right\rbrace< \eta
< \min \left\lbrace -\frac{N+k_2-2}{2}, 0 \right\rbrace. 
$$
The function $\varphie $ belongs to $H_\mu^1$ for any $\varepsilon >0$.

For this choice of $\eta$ we obtain $\eta^2 < c$,  $|x|^{2\eta}\in L_{loc}^1(\R^N,d\mu)$
and $|x|^{2\eta-2}\notin L_{loc}^1(\R^N,d\mu)$.

Let us assume that $c>c_o(N+k_2)$.   
Our aim is to prove that the bottom of the spectrum of the operator $-(L+V)$ 
\begin{equation}\label{lambda1}
\lambda_1=\inf_{\varphi \in H^1_\mu\setminus \{0\}}
\left(\frac{\int_{{\mathbb R}^N}|\nabla \varphi |^2\,d\mu
-\int_{{\mathbb R}^N}\frac{c}{|x|^2}\varphi^2\,d\mu}{\int_{{\mathbb R}^N}\varphi^2\,d\mu}
\right).
\end{equation}
is $-\infty$. For this purpose we estimate at first the numerator in (\ref{lambda1}).

\begin{equation}\label{numerator}
\begin{split}
\int_{\R^N}& \left( |\nabla \varphie|^2 -\frac{c}{|x|^2}\varphie^2 \right) \, d\mu =
\\ &=
\int_{B_1}\left[ |\nabla(\varepsilon +|x|)^\eta|^2 -\frac{c}{|x|^2}(\varepsilon +|x|)^{2\eta}\right]\, d\mu 
\\ &
+\int_{B_1^c}\left[ |\nabla(\varepsilon +|x|)^\eta \theta|^2
 -\frac{c}{|x|^2}(\varepsilon +|x|)^{2\eta}\theta^2 \right]\, d\mu
\\ & \le
\int_{B_1}\left[\eta^2(\varepsilon +|x|)^{2\eta -2}
-\frac{c}{|x|^2}(\varepsilon +|x|)^{2\eta} \right]\, d\mu\\
&+\eta^2\int_{B_1^c}  (\varepsilon +|x|)^{2\eta -2}\theta^2\, d\mu 
+ \int_{B_1^c}(\varepsilon +|x|)^{2\eta}|\nabla \theta|^2\, d\mu
\\&+
 2\eta\int_{B_1^c} \theta (\varepsilon +|x|)^{2\eta-1}\frac{x}{|x|}\cdot \nabla \theta\, d\mu
\\ & \le
 \int_{B_1} (\varepsilon +|x|)^{2\eta} \left[ \frac{\eta^2}{(\varepsilon +|x|)^2}
 -\frac{c}{|x|^2}\right]  \,d\mu
\\&+ 
2\eta^2\int_{B_1^c}(\varepsilon +|x|)^{2\eta-2}\theta^2\,d\mu 
+ 2\int_{B_1^c}(\varepsilon +|x|)^{2\eta}|\nabla \theta|^2\,d\mu
\\ & \le
\int_{B_1} (\varepsilon +|x|)^{2\eta} \left[ \frac{\eta^2}{(\varepsilon +|x|)^2} 
-\frac{c}{|x|^2}\right]  \,d\mu +C_1,
\end{split}
\end{equation}
where $C_1=\left(2\eta^2+ 2\|\nabla \theta\|_\infty \right)\int_{B_1^c}d\mu$.

Furthermore 
\begin{equation}\label{denominator}
\int_{\R^N}\varphie^2\,d\mu \ge \int_{B_2\setminus B_1 }(\varepsilon +
|x|)^{2\eta}\theta^2\,d\mu =C_{2,\varepsilon}.
\end{equation}
Putting together (\ref{numerator}) and (\ref{denominator}) we get from (\ref{lambda1})
$$
\lambda_1 \le \frac{\int_{B_1}(\varepsilon +|x|)^{2\eta}\left( \frac{\eta^2}
{(\varepsilon +|x|)^2}-\frac{c}{|x|^2} \right) \,d\mu+C_1}{C_{2,\varepsilon}}.
$$
Letting $\varepsilon\to 0$ in the numerator above, taking in mind that 
$|x|^{2\eta}\in L^1_{loc}(\R^N, d\mu)$
and Fatou's lemma, we obtain
$$
\lim\limits_{\varepsilon\to 0}\int_{B_1}(\varepsilon +|x|)^{2\eta}\left( \frac{\eta^2}
{(\varepsilon +|x|)^2}-\frac{c}{|x|^2} \right) \,d\mu
\le -(c-\eta^2)\int_{B_1}|x|^{2\eta-2}\,d\mu=-\infty
$$
and, then,  $\lambda_1=-\infty $.
\end{proof}

\bigskip

\section{Kolmogorov operators and existence and nonexistence results}

\bigskip

In the standard setting one considers $\mu\in C^{1,\lambda}_{loc}(\R^N)$ for some
$\lambda\in (0,1)$ and $\mu>0$  for any $x\in \R^N$.

We consider Kolmogorov operators
\begin{equation}\label{L}
Lu=\Delta u +\frac{\nabla \mu}{\mu}\cdot\nabla u,
\end{equation}
on smooth functions, where the probability density $\mu$ in the drift term 
is not necessarily $(1,\lambda)$-H\"olderian in the whole space 
but belongs to $C^{1,\lambda}_{loc}(\R^N\setminus\{0\})$.

These operators arise from the bilinear form integrating by parts
\begin{equation*}
a_\mu(u,v)=\int_{{\mathbb R}^N}\nabla u\cdot \nabla
v\,d\mu=-\int_{{\mathbb R}^N}(Lu)v\,d\mu.
\end{equation*}
The purpose is to get existence and nonexistence results  for weak solutions to the
the initial value problem on $L^2_\mu$ 
corresponding to the operator $L$ perturbed by an inverse square potential
$$
(P)\quad \left\{\begin{array}{ll}
\partial_tu(x,t)=Lu(x,t)+V(x)u(x,t),\quad \,x\in {\mathbb R}^N, t>0,\\
u(\cdot ,0)=u_0\geq 0\in L_\mu^2,
\end{array}
\right. $$
where $V(x)=\frac{c}{|x|^2}$, with $c>0$.

We say that $u$ is a weak solution to ($P$) if, for each $T, R>0 $, we have
$$u\in C(\left[ 0, T \right] , L^2_\mu ), \quad Vu\in L^1(B_R \times \left( 0,T\right) , d\mu dt )$$
and
\begin{equation*}
\int_0^T \int_{\R^N}u(-\partial_t\phi - L\phi )\,d\mu dt -\int_{\R^N}u_0\phi(\cdot ,0)\,d\mu =
\int_0^T \int_{\R^N} Vu\phi \, d\mu dt
\end{equation*}
for all $\phi \in W_2^{2,1}(\R^N \times \left[ 0,T\right])$ having compact support with $\phi(\cdot , T)=0$, 
where $B_R$ denotes the open ball of $\R^N$ of radius $R$ centered at $0$.
For any $\Omega\subset \R^N$, $ W_2^{2,1}(\Omega\times (0,T)) $ is the parabolic Sobolev space 
of the functions $u\in L^2(\Omega \times (0,T)) $ having weak space derivatives $D_x^{\alpha}
u\in L^2(\Omega \times (0,T))$ for $|\alpha |\le 2$ and weak time derivative
$\partial_t u \in L^2(\Omega \times (0,T))$ equipped with the norm 
\begin{equation*}
\|u\|_{W_2^{2,1}(\Omega\times (0,T))}:= \Biggl( 
\|u\|_{L^2(\Omega \times (0,T))}^2 + \|\partial_t u\|_{L^2(\Omega \times (0,T))}^2 
\, + \sum_{1\le |\alpha |\le 2} \|D^{\alpha}u\|_{L^2(\Omega \times (0,T))}^2
\Biggr)^{\frac{1}{2}}.
\end{equation*}
Let us assume that the function $\mu$ is a probability density on $\R^N$, $\mu >0$. 
In the hypothesis
\medskip
\begin{itemize}
\item[$H_7)$] $\quad\mu \in C_{loc}^{1,\lambda}(\R^N)$, $\lambda\in(0,1)$
\end{itemize}
it is known that the operator $L$ with domain
$$D_{max}(L)=\lbrace u\in C_b(\R^N)\cap W_{loc}^{2,p}(\R^N) \text{ for all }  1<p<\infty, 
Lu\in C_b(\R^N)\rbrace$$
is the weak generator of a not necessarily $C_0$-semigroup in $C_b(\R^N)$. 
Since $\int_{\R^N}Lu\,d\mu=0$ for any $u\in C_c^{\infty}(\R^N)$,
then $d\mu=\mu (x)dx$ is the invariant measure for this semigroup in $C_b(\R^N)$. 
So we can extend it to a positivity preserving and analytic $C_0$-semigroup $\lbrace T(t)\rbrace_{t\ge 0}$  on $L^2_\mu$, 
whose generator is still denoted by $L$ (see \cite{BertoldiLorenzi}).

When the assumptions on $\mu$ allow degeneracy at one point, we require the following conditions
to get $L$ generates a semigroup:

\medskip

\begin{itemize}
\item [$H_8)$] $\quad\mu \in C_{loc}^{1,\lambda}(\R^N\setminus \{0\})$, $\lambda\in(0,1)$, 
$\mu\in H^{1}_{loc}(\R^N)$,  $\frac{\nabla \mu}{\mu}\in L^r_{loc}(\R^N)$ 
for some $r>N$, and $\inf_{x\in K}\mu (x)>0$ for any compact set $K\subset \R^N$.
\end{itemize}
\medskip
So by \cite[Corollary 3.7]{alb-lor-man}), we have that the closure of 
$(L,C_c^{\infty}(\R^N))$ on $L^2_\mu$ generates
a strongly continuous and analytic Markov semigroup $\lbrace T(t)\rbrace_{t\ge 0}$  on $L^2_\mu$.

We observe that the function 
$e^{-\delta |x|^m}$ fully satisfies the condition $H_8)$
while $\cos e^{-|x|^2}$ is $(1,\lambda)$-H\"olderian in $\R^N$ (see Examples in Section 2).

For weight functions $\mu$ satisfying assumption $H_7)$ or $H_8)$ there are 
some interesting properties regarding the semigroup $\{T(t)\}_{t\ge 0}$ generated by 
the operator $L$.  These properties listed in the Proposition below
are well known under hypothesis $H_7)$ (see \cite{BertoldiLorenzi}) and 
have been proved in \cite{CGRT} if $\mu$ satisfies $H_8)$.

\medskip

\begin{prop}\label{propH1} 
Assume that $\mu$ satisfies $H_7)$ or $H_8)$. Then the following assertions hold:
\begin{itemize}
 \item [(i)] $D(L)\subset H^1_{\mu}$.
 \item [(ii)] For every $f\in D(L),\, g\in H^1_{\mu}$ we have
  \[\int L f g \,d\mu=-\int \nabla f\cdot \nabla g\,d\mu.\]
 \item [(iii)] $T(t)L^2_\mu\subset D(L)$ for all $t>0$.
  \end{itemize}
\end{prop}

\medskip

The following Theorem stated in \cite{GGR} for
functions $\mu$ satisfying condition $H_7)$,
was proved in \cite{CGRT} for functions $\mu$ under condition $H_8)$.

\medskip

\begin{thm}\label{theor as CM}
Let $0\le V(x)\in L^1_{loc}(\R^N)$. 
Assume that the weight function $\mu$ satisfies $H_4)$, $H_5)$ and $H_8)$.
Then the following assertions hold:
\begin{itemize}
\item[(i)] If $\lambda_1(L+V)>-\infty$, then there exists a
positive weak solution $u\in C([0,\infty),L^2_\mu)$ of $(P)$ satisfying
\begin{equation}\label{eq 1}
\|u(t)\|_{L^2_\mu}\le Me^{\omega t}\|u_0\|_{L^2_\mu},\quad t\ge0
\end{equation}
for some constants $M\ge 1$ and $\omega \in {\mathbb R}$.
 \item[(ii)] If
$\lambda_1(L+V)=-\infty$, then for any $0\le u_0\in
L^2_\mu\setminus \{0\},$ there is no positive weak solution of $(P)$
satisfying \eqref{eq 1}.
\end{itemize}
\end{thm}
\medskip

To get existence and nonexistence of solutions to $(P)$ we put together 
the weighted Hardy inequality (\ref{Thm wHi2}), 
Theorem \ref{thm-opt} and Theorem \ref{theor as CM}.
So we can state the following result.

\medskip

\begin{thm}\label{thm-main2}
Assume that the weight function $\mu$ satisfies hypotheses 
$H_2)$--$H_6)$, $H_8)$ and $0\le V(x)\le \frac{c}{|x|^2}$.
The following assertions hold:
\begin{enumerate}
\item[(i)] If $0\le c\le c_o(N+k_2)=\left( \frac{N+k_2-2}{2} \right)^2$, 
then there exists a positive weak
solution $u\in C([0,\infty),L^2_\mu)$ of $(P)$
satisfying
\begin{equation}\label{eq:est}
\|u(t)\|_{L^2_\mu}\le Me^{\omega t}\|u_0\|_{L^2_\mu},\quad
t\ge 0
\end{equation}
for some constants $M\ge 1$, $\omega \in {\mathbb R}$, and any $u_0\in L^2_\mu$. 
\item[(ii)] If
$c> c_o(N+k_2)$, then for any $0\le u_0\in L^2_\mu,\,u_0\neq 0,$
there is no positive weak solution of $(P)$
with $V(x)=\frac{c}{|x|^2}$ satisfying 
(\ref{eq:est}).
\end{enumerate}
\end{thm}

\bigskip

\textit{E-mail addresses: acanale@unisa.it, francesco.pappalardo@unina.it, ctarant@unina.it}


\bigskip

\begin{thebibliography}{22}


\bibitem{alb-lor-man}
A. Albanese, L. Lorenzi, E. Mangino, \textit{$L^p$–uniqueness for elliptic operators with unbounded 
coefficients in $\R^N$}, J. Funct.
Anal. 256 (2009), pp. 1238--1257.

\bibitem{Aronson}
D. G. Aronson, \textit{Non-negative solutions of linear parabolic equations}
Ann. Scuola Norm, Sup. Pisa 22 (1968), pp. 607--694.

\bibitem{BarasGoldstein} P. Baras, J. A. Goldstein, \textit{The heat equation with singular potential}, 
Trans. Am. Math. Soc. 284 (1984), pp. 121--139.

\bibitem{BarasGoldstein2} P. Baras, J. A. Goldstein, \textit{Remark on the inverse squar potential 
in quantum mechanics}, 
Internat. Conf. Diff. Eqns. (edited by I. Knowles and R. Lewis),  
North-Holland, Amsterdam, 1984.

\bibitem{BerestyckiEsteban}
H. Berestycki, M. J. Esteban, \textit{Existence and bifurcation of solutions for an elliptic degenerate problem}, 
J. Differential
Equations 134 (1) (1997), pp. 1--25.

\bibitem{B}
O. V. Besov, 
\textit{On the denseness of compactly supported functions in a weighted Sobolev
space}, Trudy Mat. Inst. Steklov. 161 (1983), pp. 29--47; English transl. in Proc. Mat. Inst. Steklov 161 (1984).

\bibitem{Bogachev}
V. I. Bogachev, 
\textit{Differentiable Measures and the Malliavin Calculus}, 
Mathematical Surveys and Monographs 164, American Mathematical Society,
Providence, Rhode Island, 2010.

\bibitem{CabreMartel}
X. Cabr\'e, Y. Martel, \textit{Existence versus explosion instantan\'ee pour des e\'quations de la chaleur line\'aires 
avec potentiel singulier}, C. R. Acad. Sci. Paris 329 (11) (1999), pp. 973--978.

\bibitem{CGRT}
A. Canale, F. Gregorio, A. Rhandi, C. Tacelli, \textit{Weighted Hardy's inequalities and Kolmogorov-type operators}, 
Appl. Anal. (2017), pp. 1--19. 


\bibitem{CP}
A. Canale, F. Pappalardo, \textit{Weighted Hardy inequalities and Ornstein-Uhlenbeck type operators perturbed 
by multipolar inverse square potentials}, J. Math. Anal. Appl. 463 (2018), pp. 895--909.

\bibitem{CPT2}
A. Canale, F. Pappalardo, C. Tarantino, \textit{Weighted multipolar Hardy inequalities and evolution problems with Kolmogorov operators perturbed by singular potentials},
in preparation.

\bibitem{CRT1}
A. Canale, A. Rhandi, C. Tacelli, 
\textit{Schr\"odinger type operators with unbounded diffusion and potential terms}, 
Annali della Scuola Normale Superiore di Pisa, Classe di Scienze, Vol. XVI (2016), pp. 581--601.

\bibitem{CRT2}
A. Canale, A. Rhandi, C. Tacelli, 
\textit{Kernel estimates for Schr\"odinger type operators with unbounded diffusion and potential terms},
Zeitschrift f\"ur Analysis und ihre Anwendungen, Vol. 36 (2016), pp. 377--392.

\bibitem{CT}
A. Canale, C. Tacelli, 
\textit{Kernel estimates for a Schr\"odinger type operator},
Rivista di Matematica della Università di Parma, Vol. 7 (2016), pp. 341--350. 

\bibitem{CPSC}
V. Chiad\'o Piat, F. Serra Cassano, 
\textit{Some Remarks About the Density of Smooth Functions in Weighted Sobolev Spaces}, 
J. Convex Anal. 1 (2) (1994), pp. 135-–142.

\bibitem{Davies} E. B. Davies, \textit{Spectral Theory and Differential Operators}, Cambridge Studies
in Advanced Mathematics 42, Cambridge University Press, 1995.


\bibitem{FKS}
E. B. Fabes, C. E. Kenig, R. P. Serapioni, 
\textit{The local regularity of solutions of degenerate elliptic equations}, Comm. Partial Differential 
Equations 7 (1982), pp. 77--116.

\bibitem{Gelfand}
I. M. Gel'fand, \textit{Some problems in the theory of quasi-linear equations}, 
Uspehi Mat. Nauk 14 (1959), pp. 87--158.

\bibitem{GGR}
G. R. Goldstein, J. A. Goldstein, A. Rhandi, \textit{Weighted Hardy's inequality and the Kolmogorov 
equation perturbed by 
an inverse-square potential}, Appl. Anal. 91 (11) (2012), pp. 2057--2071.

\bibitem{K}
A. Kufner, \textit{Weighted Sobolev spaces}, Texte zur Mathematik 31, Teubner,
Stuttgart–Leipzig–Wiesbaden, 1980.

\bibitem{KO}
A. Kufner, B. Opic, 
\textit{How to define reasonably weighted Sobolev spaces}, Comment. Math.
Univ. Carolinae 25 (3) (1984), pp. 537-–554.


\bibitem{Levy}
J. M. L\'evy-Leblond, \textit{Electron capture by polar molecules}, Phys. Rev. 153 (1) (1967), pp. 1--4.

\bibitem{BertoldiLorenzi}
L. Lorenzi, M. Bertoldi,  {\it{Analytical Methods for Markov Semigroups}}, Pure and Applied Mathematics 283, 
CRC Press, 2006.



\bibitem{T}
J. M. T\"olle, 
\textit{Uniqueness of weighted Sobolev spaces with weakly differentiable weights}, 
J. Funct. Anal. 263 (2012), pp. 3195--3223.

\bibitem{Z}
V. V. Zhikov, 
\textit{Weighted Sobolev spaces}, Sb. Math. 189 (8) (1998), pp. 1139--1170, translated
from Math. Sb. 189 (8) (1998), pp. 27--58.

\end{thebibliography}
\end{document}